\documentclass[reqno]{amsart}
\usepackage{latexsym,amsmath,amsfonts,amssymb,amsthm,color}
\theoremstyle{plain}
\newtheorem{Thm}{Theorem}
\newtheorem{Lem}{Lemma}

\theoremstyle{definition}
\newtheorem*{Ack}{Acknowledgment}
\theoremstyle{remark}

\def\Z{\mathbb Z}
\def\N{\mathbb N}

\def\L{\mathcal L}

\def\p{\mathfrak p}

\def\1{{\bf 1}}
\def\L{\mathcal L}
\def\pmod #1{\ ({\rm mod}\ #1)}
\def\mod #1{\ {\rm mod}\ #1}

\begin{document}
\medskip
\title{Stern's type congruences for $L(-k,\chi)$}
\author{Hao Pan}
\address{Department of Mathematics, Nanjing University, Nanjing 210093,
People's Republic of China}
\email{haopan79@yahoo.com.cn}
\author{Yong Zhang}
\address{Department of Basic Course, Nanjing Institute of Technology, Nanjing 211167,
People's Republic of China}
\email{yongzhang1982@163.com}
 \maketitle

\vskip 10pt

\section{Introduction}
\setcounter{Lem}{0}\setcounter{Thm}{0}\setcounter{Cor}{0}
\setcounter{equation}{0}

The Bernoulli numbers $B_k$ are given by
$$
\sum_{k=0}^{\infty}\frac{B_k}{k!}t^k=\frac{t}{e^{t}-1}.
$$
The classical Kummer theorem asserts that for a prime $p\geq 5$, $n\geq 1$ and even $k,l\geq 0$,
$$
(1-p^{k-1})\frac{B_k}{k}\equiv(1-p^{l-1})\frac{B_l}{l}\pmod{p^{n}}
$$
under the condition that $k\equiv l\pmod{\phi(p^n)}$ and $p-1\nmid k$, where $\phi$ is the Euler totient function. For a Dirichlet character $\chi$ modulo $m$, the generalized Bernoulli numbers $B_{k,\chi}$ are defined by
$$
\sum_{k=0}^{\infty}\frac{B_{k,\chi}}{k!}t^k=\sum_{a=1}^m\chi(a)\frac{te^{at}}{e^{mt}-1}.
$$
When $k$ and $\chi$ have the opposite parity (i.e., $\chi(-1)\not=(-1)^k$), it is well-known that
$$
L(-k,\chi)=-\frac{B_{k+1,\chi}}{k+1}
$$
for the non-negative integer $k$, where $L(s,\chi)$ is the Dirichlet $L$-function given by
$$
L(s,\chi)=\sum_{n=1}^\infty\frac{\chi(n)}{n^s}
$$
for $\Re s>1$. Similarly, Envall \cite{Ernvall83} also proved the Kummer-type congruence for the generalized Bernoulli numbers:
\begin{equation}\label{gbk}
(1-\chi(p)p^{k-1})\frac{B_{k,\chi}}{k}\equiv(1-\chi(p)p^{l-1})\frac{B_{l,\chi}}{l}\pmod{p^{n}}
\end{equation}
provided that  the conductor of $\chi$ is not a power of $p$ and $k\equiv l\pmod{\phi(p^n)}$. Clearly, (\ref{gbk}) can be rewritten as
$$
(1-\chi(p)p^{1-k})L(k-1,\chi)\equiv(1-\chi(p)p^{l-1})L(1-l,\chi)\pmod{p^{n}}.
$$

On the other hand, the Euler numbers (or the secant numbers) $E_k$ are given by
$$
\sum_{k=0}^{\infty}\frac{E_k}{k!}t^k=\frac{2t}{e^{t}+e^{-t}}.
$$
It is not difficult to check that
$$
E_k=2L(-k,\chi_{-4}),
$$
where $\chi_{-4}$ is the unique non-trivial character modulo $4$. Hence for an odd prime $p$, we have
\begin{equation}\label{ek}
E_{k}\equiv E_l\pmod{p},
\end{equation}
provided that $k,l\geq 0$ are even and $k\equiv l\pmod{p-1}$. In fact, (\ref{ek}) was also obtained by Kummer. However, for $E_n$, the case $p=2$ is a little special. An old result of Stern says that
for even $k,l\geq 0$,
$$
E_{k}\equiv E_l\pmod{2^n}
$$
if and only if $k\equiv l\pmod{2^n}$. Clearly Stern's result can be restated as the congruence
\begin{equation}\label{stern}
E_{k+2^nq}\equiv E_{k}+2^n\pmod{2^{n+1}}
\end{equation}
for even $k\geq0$ and odd $q\geq 1$. For the proofs and extensions of Stern's congruence, the readers may refer to \cite{Wagstaff02}, \cite{Sun05} and \cite{Sun10}.

Note that the Kummer type congruence (\ref{gbk}) requires $\chi$ is not a character modulo the power of $p$. So it is natural ask whether there exists the Stern type congruence for $L(-k,\chi)$, provided that the conductor of $\chi$ is a power of $p$. The main purpose of this paper is to establish such congruence. Suppose that $p$ is a prime and $\chi$ is a character with the conductor $p^m$. Define
$$
\L_{k,\chi}=\begin{cases}(1-\chi(5))L(-k,\chi),&\text{if }p=2\text{ and }m\geq 3,\\
(1-\chi(p+1))L(-k,\chi),&\text{if }p\geq 3\text{ and }m\geq 2,\\\end{cases}
$$
\begin{Thm}\label{sc2} Suppose that $m\geq 3$ and $\chi$ be a primitive character $\chi$ modulo $2^m$. Let  $q\geq 1$ be odd and $n\geq 1$. If $k\geq 0$ has the opposite parity as $\chi$,
then
\begin{equation}
\L_{k+2^nq,\chi}-\L_{k,\chi}\equiv\frac{2^{n+2}}{1-\bar{\chi}(5)}\cdot\L_{d,\chi}\pmod{2^{n+3}},
\end{equation}
where $d\in\{0,1\}$ satisfying $k\equiv d\pmod{2}$. In particular, we have
\begin{equation}
\L_{k,\chi}\equiv\L_{l,\chi}\pmod{2^{n+2}}
\end{equation}
if and only if $k\equiv l\pmod{2^n}$.
\end{Thm}
Similarly, for odd prime $p$, we have
\begin{Thm}\label{scp}
Let $p$ be an odd prime and $\chi$ be a primitive character $\chi$ modulo $p^m$, where $m\geq 1$. Let $n$ and $q$ be positive integers with $p\nmid q$. Suppose that $k\geq 0$ and $k$ has the opposite parity as $\chi$.

\medskip\noindent {\rm (i)} Suppose that $1-\chi(a)a^{k+1}$ is prime to $p$ for an integer $a$ with $p\nmid a$. Then
\begin{equation}
L(-k-\phi(p^n)q,\chi)\equiv L(-k,\chi)\pmod{p^{n}}.
\end{equation}

\medskip\noindent {\rm (ii)} Suppose that $m\geq 2$ and $1-\chi(a)a^{k+1}$ is not prime to $p$ for every $a$ with $(a,p)=1$. Then
\begin{equation}
\L_{k+\phi(p^n)q,\chi}-\L_{k,\chi}
\equiv\frac{p^{n}q}{1-\bar{\chi}(p+1)}\cdot\L_{d,\chi}\pmod{p^{n}},
\end{equation}
where $d\in\{0,1,2,\ldots,p-2\}$ with $k\equiv d\pmod{p-1}$. In particular, we have
\begin{equation}
\L_{k+(p-1)h,\chi}\equiv\L_{k,\chi}\pmod{p^{n}}
\end{equation}
if and only if $h\equiv 0\pmod{p^{n-1}}$.
\end{Thm}

\section{power sums}
\setcounter{Lem}{0}\setcounter{Thm}{0}\setcounter{Cor}{0}
\setcounter{equation}{0}
Define
$$
S_k(n,\chi)=\sum_{j=1}^n\chi(j)j^k.
$$
In this section, we shall mainly study $S_{k}(p^n,\chi)$ modulo $p^n$.
\begin{Lem}\label{skn}
Suppose that $p$ is a prime and $\chi$ is a character modulo $p^m$, where $m\geq 1$.
Then for $n\geq m$ and $k\geq 0$,
\begin{equation}\label{skne}
S_k(p^n,\chi)\equiv p^{n-m}S_k(p^m,\chi)\pmod{p^n},
\end{equation}
unless $p=2$, $n=1$ and $k$ is odd.
\end{Lem}
\begin{proof}
We use induction on $n$. There is nothing to do when $n=m$. Assume that $n>m$ and the assertion holds for smaller values of $n$.
By the induction hypothesis,
\begin{align*}
&S_k(p^n,\chi)=\sum_{i=0}^{p-1}\sum_{j=1}^{p^{n-1}}\chi(p^{n-1}i+j)(p^{n-1}i+j)^k
\equiv\sum_{i=0}^{p-1}\sum_{j=1}^{p^{n-1}}\chi(j)(kp^{n-1}ij^{k-1}+j^k)\\
=&pS_k(p^{n-1},\chi)+kp^{n-1}S_{k-1}(p^{n-1},\chi)\sum_{i=0}^{p-1}i\equiv p^{n-m}S_k(p^m,\chi)+\frac{kp^n(p-1)}{2}S_{k-1}(p^{n-1},\chi)\pmod{p^n}.
\end{align*}
Hence our assertion clearly holds for odd prime $p$. When $p=2$, we also have
$$
S_{k-1}(2^{n-1},\chi)=\sum_{j=1}^{2^{n-1}}\chi(j)j^{k-1}\equiv \sum_{j=1}^{2^{n-1}}\chi(j)=0\pmod{2}.
$$
\end{proof}
Thus we only need to consider $S_k(p^m,\chi)\mod p^m$.
\begin{Lem}\label{skn0}
Suppose that $p$ is a prime and $\chi$ is a primitive character modulo $p^m$. Suppose that $a$ is prime to $p$.
Then for $k\geq 0$,
\begin{equation}\label{skn0e}
(1-\chi(a)a^k)S_k(m,\chi)\equiv 0\pmod{p^m}.
\end{equation}
\end{Lem}
\begin{proof} Clearly our assertion immediately follows from the fact
$$
\sum_{j=1}^m\chi(j)j^k\equiv\sum_{j=1}^m\chi(ja)(ja)^k\equiv \chi(a)a^{k}\sum_{j=1}^m\chi(j)j^k\pmod{p^m}.
$$
\end{proof}
Consequently, for $n\geq m$, we always have
\begin{equation}S_k(p^n,\chi)\equiv 0\pmod{p^n},\end{equation}
provided that $1-\chi(a)a^k$ is prime to $p$ for some $a$.
\begin{Lem}\label{pru}
Suppose that $p$ is a prime and $\chi$ is a Dirichlet character with the conductor $p^m$.

(i) If $p$ is odd and $m>k\geq 1$, then $\chi(p^{m-k}+1)$ is a $p^k$-th primitive root of unity.

(ii) If $p=2$ and $m\geq 3$, then $\chi(5)\not=1$ is a $2^{m-2}$-th root of unity.
\end{Lem}
\begin{proof} (i) First, we show that $\chi(p^{m-1}+1)$ is a $p$-th primitive root of unity. Clearly,
$$
\chi(p^{m-1}+1)^p=\chi((p^{m-1}+1)^p)=\chi\bigg(1+\sum_{j=1}^p\binom{p}{j}p^{j(m-1)}\bigg)=\chi(1)=1.
$$
So we only need to show that $\chi(p^{m-1}+1)\not=1$. Assume on the contrary that $\chi(p^{m-1}+1)=1$.
Then $\chi(kp^{m-1}+1)=\chi((p^{m-1}+1)^k)=1$ for each $0\leq k\leq p-1$. Thus for any $0\leq k\leq p-1$ and $1\leq d\leq p^{m-1}-1$ with $p\nmid d$, letting $0\leq k'\leq p-1$ be the integer such that $dk'\equiv k\pmod{p}$, we have
$$
\chi(kp^{m-1}+d)=\chi(d(k'p^{m-1}+1))=\chi(d)\chi(k'p^{m-1}+1))=\chi(d).
$$
Hence $\chi$ can be reduced to a character modulo $p^{m-1}$. This leads an contradiction since the conductor of $\chi$ is $p^m$.

Suppose that $k\geq 2$. Similarly, we have $\chi(p^{m-k}+1)$ is a $p^k$-th primitive root of unity. Note that
$$
(p^{m-k}+1)^{p^{k-1}}=1+p^{m-1}+\sum_{r=2}^{p^{k-1}}\binom{p^{k-1}-1}{r-1}\frac{p^{m-1+(m-k)(r-1)}}{r}\equiv p^{m-1}+1\pmod{p^m},
$$
since $p^{r-1}$ can't divides $r$ for $r\geq 2$.
Then we have $\chi(p^{m-k}+1)^{p^{k-1}}\not=1$, i.e., $\chi(p^{m-k}+1)$ is a $p^k$-th root of unity.

(ii) Since $a^{2^{m-2}}\equiv 1\pmod{2^m}$ for any odd $a$, we have $\chi(5)^{2^{m-2}}=\chi(5^{2^{m-2}})=1$. Assume that $\chi(5)=1$. Note that
for every $a\equiv 1\pmod{4}$, there exist $0\leq j\leq 2^{m-2}-1$ such that $a\equiv 5^j\pmod{2^m}$. Hence $\chi(a)=1$ for every $a\equiv 1\pmod{4}$. Thus $\chi$ can be reduced to a character modulo $4$. So we must have $\chi(5)\not=1$.
\end{proof}
Suppose that the conductor of $\chi$ is $p^m$. If $m=1$, by (\ref{skne}), clearly we have $S_k(p^n,\chi)$ is divisible by $p^{n-1}$.
If $p$ is odd and $m\geq 2$, then by Lemma \ref{pru}, $1-\chi(p+1)$ is a proper divisor of $p$.
So substituting $a=p+1$ in (\ref{skn0}), we also have $p^{n-1}$ divides $S_k(p^n,\chi)$. Thus we can get
\begin{Lem}\label{skn01} Suppose that $p$ is a prime and $\chi$ is a primitive character modulo $p^m$.
Then for $n\geq m$ and $k\geq 0$,
\begin{equation}\label{skpm1}S_k(p^n,\chi)\equiv0\pmod{p^{n-1}}.
\end{equation}
Furthermore,
if $m\geq 3$, or $m=2$ and $k$ is even, or $m=1$ and $k$ is odd, then for $n\geq\max\{m,2\}$,
\begin{equation}\label{skn01b}
S_k(2^n,\chi)\equiv0\pmod{2^n}.
\end{equation}
\end{Lem}
\begin{proof} By the above discussions, we only need to consider the second assertion of this lemma, i.e., the case  $p=2$.
If $m\geq 3$, then
\begin{align*}
&S_k(2^m,\chi)=\sum_{j=1}^{2^{m}}\chi(j)j^k
\equiv\sum_{i=0}^1\sum_{j=0}^{2^{m-2}-1}\chi((-1)^i5^j)(-1)^{ik}5^{jk}\\
=&\bigg(\sum_{i=0}^1\chi(-1)^i(-1)^{ik}\bigg)\bigg(\sum_{j=0}^{2^{m-2}-1}\chi(5)^j5^{jk}\bigg)\\
=&(1+\chi(-1)(-1)^k)\cdot\frac{1-5^{2^{m-2}k}}{1-\chi(5)5^k}\pmod{2^m}.
\end{align*}
If $\chi(-1)=(-1)^{k-1}$, there is nothing to do. Suppose that $\chi(-1)=(-1)^{k}$.
Clearly $$5^{2^{m-2}k}=(4+1)^{2^{m-2}k}\equiv 1\pmod{2^m}.$$
By Lemma \ref{pru},
$\chi(5)\not=1$ is a $2^{m-2}$-th root of unity, i.e., $1-\chi(5)$ divides $2$. So
$2(\chi(5)5^k-1)^{-1}$
is $2$-integral, by noting that
$$
1-\chi(5)5^k\equiv1-\chi(5)\pmod{4}.
$$
Thus we get that $2^m$ divides $S_k(2^m,\chi)$.

If $m=2$, then clearly $S_k(2^2,\chi)\equiv 0\pmod{2}$. And when $k$ is even, we have
$$
S_k(2^2,\chi)=1+\chi(3)3^k\equiv1+(-1)\cdot(-1)^k=0\pmod{4}.
$$
Finally, when $m=1$ and $k$ is odd, by Lemma \ref{skn}, we have
$$
S_k(2^n,\chi)\equiv 2^{n-2}S_k(2^2,\chi)=2^{n-2}(1+3^k)\equiv 0\pmod{2^n}.
$$
\end{proof}

\section{Voronoi's type congruence for $L(-n,\chi)$}
\setcounter{Lem}{0}\setcounter{Thm}{0}\setcounter{Cor}{0}
\setcounter{equation}{0}

In 1889, Voronoi proved the following congruence for the Bernoulli numbers:
$$
(a^k-1)B_k\equiv ka^{k-1}\sum_{j=1}^{p-1}j^{k-1}\bigg\lfloor\frac{ja}{p}\bigg\rfloor\pmod{p},
$$
provided $k$ is even and $p\nmid a$, where $\lfloor x\rfloor=\max\{z\in\Z:\,z\leq x\}$. In \cite{Sun05}, Z.-W. Sun proved a similar congruence for the Euler numbers:
\begin{equation}\label{sv}
\frac{3^{k+1}+1}{4}E_k\equiv\frac{3^{k}}{2}\sum_{j=0}^{2^{n}-1}(-1)^{j-1}(2j+1)^k\bigg\lfloor\frac{3j+1}{2^n}\bigg\rfloor\pmod{2^n}
\end{equation}
for even $k\geq 0$. He also showed that the Stern congruence (\ref{stern}) is an easy consequence of (\ref{sv}). Now we are ready to extend (\ref{sv}) for $L(-k,\chi)$.
\begin{Thm} \label{ab} Let $p$ be a prime and $a$ be an integer with $p\nmid a$. Suppose that $\chi$ is a character modulo $p^{m}$. Suppose that $n\geq m$, $k\geq 0$ and $k$ has the opposite parity as $\chi$. If $p\geq 5$, or $p=2,3$ and $n\geq 2$, Then
\begin{equation}\label{t1e2}
(1-\chi(a)a^{k+1})L(-k,\chi)\equiv
\chi(a)a^{k}\sum_{j=1}^{p^{n}-1}\chi(j)
j^{k}\bigg\lfloor\frac{ja}{p^{n}}\bigg\rfloor
\pmod{p^{n}}.
\end{equation}
\end{Thm}
\begin{proof} For $1\leq j<p^{n}$, we write
$$ja=p^{n}\bigg\lfloor\frac{ja}{p^{n}}\bigg\rfloor+r_j,$$ where $0\leq r_j
<p^{n}$. For a positive integer $d$, define $\nu_p(d)=\max\{\alpha\in\N:\,p^\alpha\mid d\}$.
Then
\begin{align*}
&\chi(a)a^kS_k(p^n,\chi)
=\sum_{j=1}^{p^{n}-1}\chi(ja)(ja)^{k}=
\sum_{j=1}^{p^{n}-1}\chi(ja)\bigg(p^{n}\bigg\lfloor\frac{ja}{p^{n}}\bigg\rfloor+r_{j}\bigg)^{k}\\
\equiv&\sum_{j=1}^{p^{n}-1}\chi({r_{j}})\sum_{t=0}^{k}\binom{k}{t}p^{tn}\bigg\lfloor\frac{ja}{p^{n}}\bigg\rfloor^t r_{j}^{k-t}\equiv\sum_{j=1}^{p^{n}-1}\chi({r_{j}})\bigg(r_{j}^{k}+\sum_{t=1}^{k}\frac{k}{t}\binom{k-1}{t-1}p^{tn}\bigg\lfloor\frac{ja}{p^{n}}\bigg\rfloor^t r_{j}^{k-t}\bigg)\\
\equiv&\sum_{j=1}^{p^{n}-1}\chi(j)j^k+kp^n\sum_{j=1}^{p^{n}-1}\chi(ja)
\bigg\lfloor\frac{ja}{p^{n}}\bigg\rfloor (ja)^{k-1}+\frac{k(k-1)}{2}p^{2n}\sum_{j=1}^{p^{n}-1}\chi(ja)
\bigg\lfloor\frac{ja}{p^{n}}\bigg\rfloor^2 (ja)^{k-2}\pmod{p^{2n+\nu_p(k)}},
\end{align*}
by noting that $\nu_p(t)\leq t-2$ for $t\geq 3$.
Thus substituting $k$ by $k+1$, we get
\begin{align}\label{p1}
\frac{\chi(a)a^{k+1}-1}{p^n}\cdot\frac{S_{k+1}(p^n,\chi)}{k+1}
\equiv\chi(a)a^{k}\sum_{j=1}^{p^{n}-1}\chi(j)
j^{k}\bigg\lfloor\frac{ja}{p^{n}}\bigg\rfloor+\frac{k}{2}\cdot p^{n}\sum_{j=1}^{p^{n}-1}\chi(ja)
\bigg\lfloor\frac{ja}{p^{n}}\bigg\rfloor^2 (ja)^{k-1} \pmod{p^{n}}.
\end{align}
In particular, if $p$ is odd, then
\begin{align}\label{po1}
\frac{\chi(a)a^{k+1}-1}{p^n}\cdot\frac{S_{k+1}(p^n,\chi)}{k+1}
\equiv\chi(a)a^{k}\sum_{j=1}^{p^{n}-1}\chi(j)
j^{k}\bigg\lfloor\frac{ja}{p^{n}}\bigg\rfloor\pmod{p^{n}}.
\end{align}

On the other hand, we know (cf. \cite[Eq. 3.2(5)]{UrbanowiczWilliams00})
\begin{align}&p^{n}B_{k+1,\chi}=S_{k}(p^{n},\chi)+\sum_{t=1}^{k+1}\binom{k+1}{t}p^{tn}B_tS_{k+1-t}(p^{n},\chi)\notag\\
=&S_{k+1}(p^{n},\chi)-\frac{(k+1)p^{n}S_{k}(p^{n},\chi)}{2}+\frac{k(k+1)p^{2n}S_{k-1}(p^{n},\chi)}{12}+\sum_{t=4}^{k+1}\frac{k+1}{t}
\binom{k}{t-1}p^{tn}B_tS_{k+1-t}(p^{n},\chi).\end{align}
In view of Lemma \ref{skn01}, when $p\geq 5$, or $p=2,3$ and $n\geq 2$, we have
$$
p^{n}B_{k+1,\chi}\equiv S_{k+1}(p^{n},\chi)-\frac{(k+1)p^{n}}{2}\cdot S_{k}(p^{n},\chi)+\frac{k(k+1)}{12}\cdot p^{2n}S_{k-1}(p^{n},\chi)\pmod{p^{2n+\nu_p(k+1)}},
$$
since $pB_{t}$ is $p$-integral by the von Staudt-Clausen theorem. That is,
\begin{equation}\label{p2}
L(-k,\chi)\equiv -\frac1{p^n}\cdot\frac{S_{k+1}(p^{n},\chi)}{k+1}+\frac{S_{k}(p^{n},\chi)}{2}-\frac{k}{12}\cdot p^{n}S_{k-1}(p^{n},\chi)\pmod{p^{n}}.
\end{equation}
\noindent{\rm (i)} Suppose that $p\geq 5$, or $p=3$ and $n\geq 2$. Since $\chi(-1)\cdot(-1)^{k}-1=-2$ is not divisible by $p$, by Lemma \ref{skn0},
$$
S_{k}(p^{n},\chi)\equiv 0\pmod{p^{n}}.
$$
And by Lemma \ref{skn01}, $3\mid S_{k-1}(3^{n},\chi)$ if $n\geq 2$.
Thus from (\ref{p2}), it follows that
\begin{align*}
(1-\chi(a)a^k)L(-k,\chi)\equiv\frac{\chi(a)a^{k+1}-1}{p^{n}}\cdot\frac{S_{k+1}(p^{n},\chi)}{k+1}\equiv\chi(a)a^{k}\sum_{j=1}^{p^{n}-1}\chi(j)
j^{k}\bigg\lfloor\frac{ja}{p^{n}}\bigg\rfloor\pmod{p^{n}}.
\end{align*}

\medskip
\noindent{\rm (ii)} Now let $p=2$.
It follows from (\ref{p1}) that
$$
\frac{\chi(a)a^{k}-1}{2^n}\cdot S_{k}(2^n,\chi)\equiv k\sum_{j=1}^{2^{n}-1}\chi(ja)
\bigg\lfloor\frac{ja}{2^{n}}\bigg\rfloor (ja)^{k-1}
\pmod{2}.
$$
Applying (\ref{p1}) again, we deduce that
\begin{align*}
&\frac{\chi(a)a^{k+1}-1}{2^{2n-1}}\cdot\frac{S_{k+1}(2^n,\chi)}{k+1}-\frac{1}{2^{n-1}}\sum_{j=1}^{2^{n}-1}\chi(ja)
\bigg\lfloor\frac{ja}{2^{n}}\bigg\rfloor (ja)^{k}\\
\equiv&\sum_{j=1}^{2^{n}-1}\chi(ja)
\bigg\lfloor\frac{ja}{2^{n}}\bigg\rfloor^2 (ja)^{k-1}
\equiv k\sum_{j=1}^{2^{n}-1}\chi(ja)
\bigg\lfloor\frac{ja}{2^{n}}\bigg\rfloor (ja)^{k-1}\equiv\frac{\chi(a)a^{k}-1}{2^n}\cdot S_{k}(2^n,\chi)\pmod{2}.
\end{align*}
Since $k$ and $\chi$ have the opposite parity and $\chi$ is primitive,
we must have that $k$ and $m$ have the same parity if $1\leq m\leq 2$.
So in view of (\ref{skn01b}), we get that
$S_{k}(2^n,\chi)$ is divisible by $2^n$. Thus
$$
\frac{\chi(a)a^{k}-1}{2^n}\cdot S_{k}(2^n,\chi)\equiv\frac{\chi(a)a^{k+1}-1}{2^n}\cdot S_{k}(2^n,\chi)\pmod{2}.
$$
Consequently,
$$
\frac{\chi(a)a^{k+1}-1}{2^{n}}\cdot\frac{S_{k+1}(2^n,\chi)}{k+1}-\frac{\chi(a)a^{k+1}-1}{2}\cdot S_{k}(2^n,\chi)\equiv\sum_{j=1}^{2^{n}-1}\chi(ja)
\bigg\lfloor\frac{ja}{2^{n}}\bigg\rfloor (ja)^{k}\pmod{2^n}.
$$
Furthermore, we have
$$
kS_{k-1}(2^n,\chi)\equiv 0\pmod{2^n}\equiv0\pmod{4}.
$$
Therefore in view of (\ref{p2}),
\begin{align*}
(1-\chi(a)a^{k+1})L(-k,\chi)\equiv&\frac{\chi(a)a^{k+1}-1}{2^n}\cdot\bigg(\frac{S_{k+1}(p^{n},\chi)}{k+1}-S_{k}(p^{n},\chi)
+\frac{k}{12}\cdot2^{n}S_{k-1}(p^{n},\chi)\bigg)\\
\equiv&\chi(a)a^{k-1}\sum_{j=1}^{2^{n}-1}\chi(j)
j^{k-1}\bigg\lfloor\frac{ja}{2^{n}}\bigg\rfloor\pmod{2^{n}}.
\end{align*}
\end{proof}

\section{Proof of Theorems \ref{sc2} and \ref{scp}}
\setcounter{Lem}{0}\setcounter{Thm}{0}\setcounter{Cor}{0}
\setcounter{equation}{0}

\begin{proof}[Proof of Theorem \ref{scp}]
Let $n_*=\max\{n+1,m\}$. Applying Theorem \ref{ab}, we have
\begin{align}\label{scpc1}
&\frac{1-\chi(a)a^{k+\phi(p^n)q+1}}{\chi(a)a^{k+\phi(p^n)q}}\cdot L(-k-\phi(p^n)q,\chi)-\frac{1-\chi(a)a^{k+1}}{\chi(a)a^{k}}\cdot L(-k,\chi)\notag\\
\equiv&\sum_{\substack{1\leq j\leq p^{n_*}\\ (j,p)=1}}\chi(j)j^{k}\cdot(j^{\phi(p^n)q}-1)\cdot\bigg\lfloor\frac{ja}{p^{n_*}}\bigg\rfloor\equiv0\pmod{p^{n}}.
\end{align}
Thus (i) of Theorem \ref{scp} is immediately derived provided that $1-\chi(a)a^{k+1}$ is prime to $p$.

Now assume that $m\geq 2$ and $1-\chi(a)a^{k+1}$ is not prime to $p$ for every $1\leq a\leq p-1$.
By Lemma \ref{pru}, $\chi(p+1)$ is a $p^{m-1}$-th primitive root of unity. So we know that
$$
\frac{(1-\chi(p+1))^{p^{m-2}(p-1)}}{p}
$$
is a unit of the $p^{m-1}$-th cyclotomic field. For $k\geq 1$,
$$
1-\chi(p+1)(p+1)^{-(k+1)}\equiv 1-\chi(p+1)\pmod{p}.
$$
Hence $1-\chi(p+1)$ divides $p$, but $(1-\chi(p+1))^2$ doesn't divide $p$.
It follows that $(1-\chi(p+1))L(-k,\chi)$ is always $p$-integral.

Let $a$ be an integer such that
$$
(p+1)a\equiv1\pmod{p^{n_*+2}}.
$$
Obviously $a-1$ is divisible by $p$, but not divisible by $p^2$.
Noting that
$$
(p+1)^{\phi(p^n)q}=\sum_{j=0}\binom{p^{n-1}(p-1)q}{j}p^j\equiv1+\cdot p^{n}(p-1)q\equiv1-p^{n}q\pmod{p^{n+1}},
$$
we get that
\begin{align}\label{scpc2}
&\frac{1-\chi(a)a^{k+\phi(p^n)q+1}}{\chi(a)a^{k+\phi(p^n)q}}\cdot L(-k-\phi(p^n)q,\chi)-\frac{1-\chi(a)a^{k+1}}{\chi(a)a^{k}}\cdot L(-k,\chi)\notag\\
\equiv&\frac{1-\chi(p+1)(p+1)^{k+1}}{p+1}\cdot L(-k,\chi)-\frac{1-\chi(p+1)(p+1)^{k+\phi(p^n)q+1}}{p+1}\cdot L(-k-\phi(p^n)q,\chi)\notag\\
\equiv&\frac{1-\chi(p+1)(p+1)^{k+\phi(p^n)q+1}}{p+1}\cdot (L(-k,\chi)-L(-k-\phi(p^n)q,\chi))-\frac{\chi(p+1)(p+1)^{k+1}p^{n}q}{p+1}\cdot L(-k,\chi)\notag\\
&\pmod{p^{n}}.
\end{align}
Furthermore, we have
\begin{align}\label{scpc3}
(1-\chi(a))(L(-k,\chi)-L(-d,\chi))\equiv\frac{(1-\chi(a))\chi(a)a^{k}}{1-\chi(a)a^{k+1}}\sum_{j=1}^{p^{m}-1}\chi(j)
(j^{k}-j^d)\bigg\lfloor\frac{ja}{p^{m}}\bigg\rfloor
\equiv0\pmod{\frac{p}{1-\chi(a)}}.
\end{align}
Hence, combining (\ref{scpc1}), (\ref{scpc2}) and (\ref{scpc3}), we obtain that
\begin{align*}
&(1-\chi(p+1))(L(-k,\chi)-L(-k-\phi(p^{n})q,\chi))\\
\equiv&\frac{(p+1)(1-\chi(p+1))}{1-\chi(p+1)(p+1)^{k+\phi(p^{n})q+1}}\cdot\frac{\chi(p+1)(p+1)^{k+1}p^{n}q}{p+1}\cdot L(-k,\chi)\\
\equiv&\chi(p+1)p^{n+1}q\cdot L(-k,\chi)\equiv\chi(p+1)p^{n}q\cdot L(-d,\chi)\pmod{p^{n}}.
\end{align*}

Finally, we need to show that $(1-\chi(p+1))L(-d,\chi)$ is not divisible by $p$.
Let $g$ be a primitive root of $p^{m}$. Since $1-\chi(g)g^{d+1}$ is not prime to $p$, we may assume that $\p$ is a prime ideal in the algebraic integers ring of the $\phi(p^m)$-th cyclotomic field, dividing both
$1-\chi(g)g^{d+1}$ and $p$. It suffices to show that
$$
\sum_{j=1}^{p^{m}-1}\chi(j)j^{d}\bigg\lfloor\frac{j(p+1)}{p^{m}}\bigg\rfloor\not\equiv 0\pmod{\p}.
$$
Note that $\p$ divide $1-\chi(g)g^{d+1}$ implies that $\p$ divides $1-\chi(g^s)(g^s)^{d+1}$ for every $s$. So we have $\p$ divides $1-\chi(j)j^{d+1}$ for any $j$ with $p\nmid j$.
It follows that
$$
\sum_{j=1}^{p^{m}-1}\chi(j)j^{d}\bigg\lfloor\frac{j(p+1)}{p^{m}}\bigg\rfloor\equiv
\sum_{\substack{1\leq j\leq p^{m}-1\\ (j,p)=1}}\frac1{j}\bigg\lfloor\frac{j(p+1)}{p^{m}}\bigg\rfloor\equiv\frac{(p+1)((p+1)^{p^{m-1}(p-1)}-1)}{p^m}
\equiv-1\pmod{\p},
$$
where we use the following known result (cf. \cite{Lerch05}):
$$
\frac{a^{\phi(n)}-1}{n}\equiv\frac1{a}\sum_{\substack{1\leq j\leq n\\ (j,n)=1}}\frac1{j}\bigg\lfloor\frac{ja}{n}\bigg\rfloor\pmod{n}
$$
where $(a,n)=1$
\end{proof}

\begin{proof}[Proof of Theorem \ref{sc2}] Since $m\geq 3$, $\chi(5)\not=1$ is a $2^{m-2}$-th primitive root of unity. So $1-\chi(5)$ divides $2$ and
$$
1-\chi(5)5^k\equiv1-\chi(5)\pmod{4}.
$$
Furthermore, applying Theorem \ref{ab} and recalling $\chi(-1)=(-1)^{k-1}$, we have
\begin{align*}
&\frac{1-\chi(5)5^{k+1}}{\chi(5)^{k}}\cdot L(-k,\chi)\\
\equiv&\sum_{j=1}^{2^{m}}\chi(j)j^{k}\bigg\lfloor\frac{5j}{2^{m}}\bigg\rfloor
=\sum_{\substack{1\leq j\leq 2^{m}\\ j\equiv1\pmod{4}}}\bigg(\chi(j)j^{k}\bigg\lfloor\frac{5j}{2^{m}}\bigg\rfloor+\chi(2^m-j)(2^m-j)^{k}\bigg\lfloor\frac{5(2^m-j)}{2^{m}}\bigg\rfloor\bigg)\\
\equiv&2\sum_{\substack{1\leq j\leq 2^{m}\\ j\equiv1\pmod{4}}}\chi(j)j^{k}\bigg(\bigg\lfloor\frac{5j}{2^{m}}\bigg\rfloor-\bigg(4-\bigg\lfloor\frac{5j}{2^{m}}\bigg\rfloor\bigg)\bigg)
\equiv2\sum_{\substack{1\leq j\leq 2^{m}\\ j\equiv1\pmod{4}}}\chi(j)j^{k}\bigg\lfloor\frac{5j}{2^{m}}\bigg\rfloor\pmod{4}.
\end{align*}
So actually $(1-\chi(5))L(-k,\chi)$ is a multiple of $2$.

According to
$$
(4j\pm1)^{2^nq}=1+\sum_{i=1}(\pm1)^i\cdot\frac{2^nq}{i}\binom{2^nq-1}{i-1}\cdot4^ij^i,
$$
we have
$$
(4j\pm1)^{2^nq}\equiv 1\pm2^{n+2}qj\pmod{2^{n+3}}.
$$

Let $n_*=\max\{n+3,m\}$ and let $a$ be an integer such that
$$
5a\equiv1\pmod{2^{n_*}}.
$$
With the help of Theorem \ref{ab},
\begin{align}\label{nk}
&\frac{1-\chi(a)a^{k+2^nq+1}}{\chi(a)a^{k+2^nq}}\cdot L(-k-2^nq,\chi)-\frac{1-\chi(a)a^{k+1}}{\chi(a)a^{k}}\cdot L(-k,\chi)\notag\\\equiv&
\sum_{\substack{1\leq j\leq 2^{n_*}\\ j\equiv1\pmod{4}}}\bigg(\chi(j)j^{k}(j^{2^nq}-1)\bigg\lfloor\frac{ja}{2^{n_*}}\bigg\rfloor+
\chi(2^{n_*}-j)(2^{n_*}-j)^{k}((2^{n_*}-j)^{2^nq}-1)\bigg\lfloor\frac{(2^{n_*}-j)a}{2^{n_*}}\bigg\rfloor\bigg)\notag\\
\equiv&\sum_{j=0}^{2^{n_*}-1}\chi(4j+1)(4j+1)^{k}\cdot
\bigg(2^{n+2}qj\bigg\lfloor\frac{(4j+1)a}{2^{n_*}}\bigg\rfloor+
2^{n+2}q(2^{n_*-2}-j)\bigg(a-1-\bigg\lfloor\frac{(4j+1)a}{2^{n_*}}\bigg\rfloor\bigg)\bigg)\notag\\
\equiv&0\pmod{2^{n+3}}.
\end{align}
On the other hand, note that
\begin{align}\label{nk1}
&\frac{1-\chi(a)a^{k+2^nq+1}}{\chi(a)a^{k+2^nq}}\cdot L(-k-2^nq,\chi)-\frac{1-\chi(a)a^{k+1}}{\chi(a)a^{k}}\cdot L(-k,\chi)\notag\\
\equiv&5^{-1}(1-\chi(5)5^{k+1})\cdot L(-k,\chi)-5^{-1}(1-\chi(5)5^{k+2^nq+1})\cdot L(-k-2^nq,\chi)\notag\\
\equiv&5^{-1}(1-\chi(5)5^{k+2^nq+1})(L(-k,\chi)-L(-k-2^nq,\chi))+\chi(5)5^{k}2^{n+2}qL(-k,\chi)\pmod{2^{n+3}}.
\end{align}
Thus by (\ref{nk}) and (\ref{nk1}), we get
\begin{align*}
(1-\chi(5))(L(-k-2^nq,\chi)-L(-k,\chi))\equiv&\frac{1-\chi(5)}{1-\chi(5)5^{k+2^nq+1}}\cdot\chi(5)5^{k+1}2^{n+2}qL(-k,\chi)\notag\\\equiv&
\chi(5)2^{n+2}L(-k,\chi)\equiv\chi(5)2^{n+2}L(-d,\chi)\pmod{2^{n+3}},
\end{align*}
where the last step follows from
$$
L(-k,\chi)\equiv\frac{\chi(a)a^{k}}{1-\chi(a)a^{k+1}}
\sum_{j=1}^{2^{m}-1}\chi(j)
j^{k}\bigg\lfloor\frac{ja}{2^{m}}\bigg\rfloor\equiv
\frac{\chi(a)a^{d}}{1-\chi(a)a^{d+1}}
\sum_{j=1}^{2^{m}-1}\chi(j)
j^{d}\bigg\lfloor\frac{ja}{2^{m}}\bigg\rfloor\equiv L(-d,\chi)
\pmod{2^{2}}.
$$

The remainder task to prove that
$$
\frac{1-\chi(5)5^{d+1}}{\chi(5)5^{d}}\cdot\frac{L(-d,\chi)}{2}\equiv
\sum_{\substack{0\leq j\leq 2^{m-2}-1}}\chi(4j+1)\bigg\lfloor\frac{5(4j+1)}{2^{m}}\bigg\rfloor\not\equiv0\pmod{1-\chi(5)}.
$$
It is easy to see that $1-\chi(5)$ divides $1-\chi(4j+1)$ for every $j$. So
\begin{align*}
&\sum_{\substack{0\leq j\leq 2^{m-2}-1}}\chi(4j+1)\bigg\lfloor\frac{5(4j+1)}{2^{m}}\bigg\rfloor\equiv
\sum_{\substack{0\leq j\leq 2^{m-2}-1}}\bigg\lfloor\frac{5(4j+1)}{2^{m}}\bigg\rfloor\\
\equiv&|\{j:\,2^m\leq 20j+5<2^{m+1}\}|+
|\{j:\,3\cdot2^m\leq 20j+5<2^{m+2}\}|\pmod{1-\chi(5)}.
\end{align*}
Note that
$$
\frac{s\cdot 2^{m+4}-5}{20}=3s\cdot 2^{m-2}+\frac{s\cdot 2^{m}-5}{20}.
$$
Thus it suffices to show that
$$
|\{j:\,2^m\leq 20j+5<2^{m+1}\}|+
|\{j:\,3\cdot2^m\leq 20j+5<2^{m+2}\}|
$$
is odd for $3\leq m\leq 6$. Of course, this can be verified directly.
\end{proof}

\begin{Ack}
We are grateful to Professors Zhi-Hong Sun and Zhi-Wei Sun for their helpful discussions on Stein's congruence.
\end{Ack}

\end{document}